\documentclass[reqno,10pt]{amsart}
\usepackage{amsmath}
\usepackage{amsthm}
\usepackage{amssymb}
\usepackage{graphicx}
\usepackage{amsfonts}
\usepackage{latexsym}
\usepackage{setspace}
\usepackage{hyperref,cite}
\usepackage{multicol}
\usepackage[active]{srcltx}
\usepackage{mathrsfs}

\newcommand{\ran}{\operatorname{ran}}
\newcommand{\cran}{\overline{\operatorname{ran}}\, }
\newcommand{\diag}{\operatorname{diag}}
\newcommand{\norm}[1]{\| #1 \|}

\newcommand{\h}{\mathcal{H}}
\newcommand{\K}{\mathcal{K}}
\newcommand{\minimatrix}[4]{\begin{bmatrix} #1 & #2 \\ #3 & #4 \end{bmatrix}  }
\newcommand{\megamatrix}[9]{\begin{bmatrix} #1 & #2 & #3 \\ #4 & #5 & #6 \\ #7 &
#8 & #9\end{bmatrix}  }

\renewcommand{\phi}{\varphi}
\renewcommand{\epsilon}{\varepsilon}

\theoremstyle{plain}
\newtheorem{Theorem}{Theorem}

\newtheorem{Lemma}[Theorem]{Lemma}
\theoremstyle{definition}
\newtheorem*{Definition}{Definition}
\newtheorem{Example}[Theorem]{Example}
\newtheorem{Question}{Question}

\allowdisplaybreaks

\begin{document}
\bibliographystyle{amsplain}

    \title[Two remarks about nilpotent operators of order two]{Two remarks about nilpotent\\operators of order two}

    \author[S. R. Garcia]{Stephan Ramon Garcia}
    \address{   Department of Mathematics\\
            Pomona College\\
            Claremont, California\\
            91711 \\ USA}
    \email{Stephan.Garcia@pomona.edu}
    \urladdr{\url{http://pages.pomona.edu/~sg064747}}
    \thanks{The first two authors are partially supported by National Science
Foundation Grant DMS-1001614.}

    \author{Bob Lutz}
    \address{   Department of Mathematics\\
            Pomona College\\
            Claremont, California\\
            91711 \\ USA}
    \email{boblutz13@gmail.com}
    
	\author[D.~Timotin]{Dan Timotin}
	\address{Simion Stoilow Institute of Mathematics of the Romanian Academy, PO Box 1-764, Bucharest 014700, Romania}
	\email{Dan.Timotin@imar.ro}


    \begin{abstract}
	    We present two novel results about Hilbert space operators which
	    are nilpotent of order two.  First, we prove that such operators are
	    \emph{indestructible} complex symmetric operators, in the sense that
	    tensoring them with any operator yields a complex symmetric operator.
	    In fact, we prove that this property characterizes nilpotents of order two
	    among all nonzero bounded operators.
	    Second, we establish that every nilpotent of order two is unitarily
	    equivalent to a truncated Toeplitz operator.
    \end{abstract}

\maketitle

\section{Introduction}

	In the following, $\h$ denotes a separable complex Hilbert space and
	$B(\h)$ denotes the set of all bounded linear operators on $\h$.  Recall that 
	an operator $T$ in $B(\h)$ is called \emph{nilpotent} if $T^n = 0$
	for some positive integer $n$.  The least such $n$ is called the
	\emph{order} of nilpotence of $T$.  This note concerns two rather
	unusual properties of operators which are nilpotent of order two.	
		
	The first result involves complex symmetric operators
	(see Section~\ref{se:prelim} for background).  It is known that
	every operator which is nilpotent of order two is a
	complex symmetric operator (Lemma~\ref{TheoremCSO}).  However,
	these operators are complex symmetric in a much stronger sense, for
	the tensor product of a nilpotent of order two with an arbitrary operator
	always yields a complex symmetric operator.  We prove in Section~\ref{se:indestructible} 
	that this property actually characterizes nilpotents of order two
	among all nonzero bounded operators.
    
	Our second result concerns truncated Toeplitz operators (precise definitions are given in
	Section \ref{se:tto}).  To be more specific, we prove that every operator which is nilpotent 
	of order two is unitarily equivalent to a truncated Toeplitz operator having an analytic symbol. This is
	relevant to a series of open problems, first arising in \cite{TTOSIUES} and developed further in
	\cite{ATTO}, which, in essence, ask whether an arbitrary complex symmetric operator is unitarily equivalent to 
	a truncated Toeplitz operator (or possibly a direct sum of such operators).
			
	We close the paper with several open questions suggested by these results.

\section{Complex symmetry}\label{se:prelim}

	Before proceeding, let us recall a few basic definitions \cite{G-P,G-P-II,CCO}.  
	A \emph{conjugation} on a complex Hilbert space $\h$ is a
	conjugate-linear, isometric involution.  We say that an operator 
	$T$ in $B(\h)$ is \emph{complex symmetric} if there exists a
	conjugation $C$ on $\h$ such that $T=CT^*C$.  In this case, 
	we say that $T$ is \emph{$C$-symmetric}.  The terminology 
	reflects the fact that an operator is complex symmetric if and only if 
	it has a self-transpose matrix representation with respect to some 
	orthonormal basis.  In fact, any orthonormal basis which is
	\emph{$C$-real}, in the sense that each basis vector is fixed by
	$C$, yields such a matrix representation.

	It is known that if $T$ is nilpotent of order two, then
	$T$ is a complex symmetric operator.  This was first established
	directly in \cite[Thm.~5]{ATCSO}, using what
	we now recognize as a somewhat overcomplicated argument.  
	Later on, this result was obtained 
	as a corollary of the more general fact that every binormal operator 
	is complex symmetric \cite[Thm.~2, Cor.~4]{SNCSO}.
	A much simpler direct proof is provided below. In what follows,
	we denote unitary equivalence by $\cong$.

	\begin{Lemma}\label{TheoremCSO}
		If $T$ is nilpotent of order two, then $T$ is a
		complex symmetric operator.
	\end{Lemma}

	\begin{proof}
		If $T$ in $B(\h)$ satisfies $T(Tx) = T^2 x = 0$ for every $x$ in $\h$, it follows that
		$\cran T \subseteq \ker T = \ker|T|$.  Considering the polar
		decomposition of $T$, we see that
		\begin{equation}\label{eq:form nilp 2}
			T \cong \minimatrix{0}{0}{A}{0} \oplus 0
		\end{equation}
		where $A$ is a positive operator with dense range (the zero
		direct summand, which acts on $\ker T \ominus \cran T$, may be 
		absent).  Without loss of generality, we may  assume that  
		$\ker T = \cran T$.  If $J$ is any conjugation
		which commutes with $A$ (the existence of such a $J$ follows
		immediately from the Spectral Theorem), we find that
		\begin{equation*}
			 \underbrace{ \minimatrix{0}{0}{A}{0} }_T
			 =
			\underbrace{ \minimatrix{0}{J}{J}{0} }_C
			\underbrace{ \minimatrix{0}{A}{0}{0} }_{T^*}
			\underbrace{ \minimatrix{0}{J}{J}{0} }_C,
		\end{equation*}
		whence $T$ is a complex symmetric operator.
	\end{proof}

	For operators on a finite dimensional space, there is a quite explicit proof. 
	Indeed, the positive semidefinite matrix $A$ is unitarily equivalent to a diagonal
	matrix $D=\diag(\lambda_1,\lambda_2,\ldots,\lambda_n)$ where 
	$\lambda_1\geq \lambda_2\geq \cdots \geq \lambda_n \geq 0$,
	whence
	\begin{equation*}
		\minimatrix{0}{0}{A}{0}
		\cong \minimatrix{0}{0}{D}{0}
		\cong \bigoplus_{i=1}^n \minimatrix{0}{0}{\lambda_i}{0}
		\cong \bigoplus_{i=1}^n \frac{\lambda_j}{2} \minimatrix{1}{i}{i}{-1}.
	\end{equation*}

\section{Indestructible complex symmetric operators}\label{se:indestructible}

	In the following, $\h$ and $\K$ denote separable complex Hilbert spaces while $A$ and $B$
	are bounded operators on $\h$ and $\K$, respectively.  Recall that the operator $A \otimes B$
	acts on the space $\h \otimes \K$ and satisfies
	\begin{equation}\label{eq-TensorNorm}
		\norm{A\otimes B}_{\h\otimes \K} = \norm{A}_{\h}\norm{B}_{\K},
	\end{equation}
	the subscripts being suppressed in practice.  The following relevant lemma is from
	\cite[Sect.~10]{G-P}, where it is stated without proof.  
		
	\begin{Lemma}
		The tensor product of complex symmetric operators is complex symmetric.
	\end{Lemma}
	
	\begin{proof}
		Suppose that $A$ and $B$ are operators and that $C$ and $J$ are 
		conjugations on $\h$ and $\K$, respectively, such that
		$A = CA^*C$ and $B = JB^*J$.
		Let $u_i$ and $v_j$ denote $C$-real and $J$-real orthonormal 
		bases of $\h$ and $\K$, respectively.
		Define a conjugation $C\otimes J$ on $\h\otimes \K$ by first setting 
		$(C\otimes J)(u_i \otimes v_j) = u_i \otimes v_j$
		on the orthonormal basis $u_i \otimes v_j$ of $\h\otimes \K$ and then extending
		this to $\h\otimes \K$ by conjugate-linearity and continuity.  One can then 
		check that $A\otimes B$ is $(C\otimes J)$-symmetric.
	\end{proof}

	On the other hand, it is possible for $A \otimes B$ to be 
	complex symmetric even if neither $A$ nor $B$ is complex symmetric.
	The following lemma provides a simple method for constructing such examples.

	\begin{Lemma}\label{le:conj_tens_prod}
		For each $A$ in $B(\h)$ and each conjugation $J$ on $\h$, 
		the operator $T = A\otimes JA^*J$ is complex symmetric.
	\end{Lemma}
	
	\begin{proof}
		If $\Phi$ in $B(\h\otimes \h)$ is defined first on simple tensors by 
		$\Phi(x\otimes y)=y\otimes x$ and then extended to $\h\otimes \h$ in the natural way,
		then $C=\Phi(J\otimes J)$ is a conjugation on $\h\otimes \h$
		with respect to which $T=CT^*C$.	
	\end{proof}

	\begin{Example} 
		Suppose $\h=\ell^2$, $A=S$ (the unilateral shift), and $J$ is
		entry-by-entry complex conjugation on $\ell^2$.
		There are many ways to see that $S$ is not a complex symmetric operator
		\cite[Cor.~7]{MUCFO}, \cite[Prop.~1]{G-P}, \cite[Ex.~2.14]{CCO},
		\cite[Thm.~4]{SNCSO}.  On the other hand, 
		Lemma~\ref{le:conj_tens_prod} implies that
		$S \otimes S^*$ is complex symmetric.  In fact, 
		\begin{equation*}
			S\otimes S^* \cong \bigoplus_{n=0}^{\infty} J_n(0),
		\end{equation*}
		where $J_n(0)$ denotes a $n \times n$ nilpotent Jordan block,
		which is complex symmetric by \cite[Ex.~4]{G-P}.  To see this,
		note that $S\otimes S^*$ is unitarily equivalent to the operator
		\begin{equation*}
			[Tf](z,w) = z\left(\frac{f(z,w) - f(z,0)}{w}\right)
		\end{equation*}
		on the Hardy space $H^2_2$ on the bidisk and observe that
		$H_2^2 = \bigoplus_{n=0}^{\infty} \mathcal{P}_n$ where $\mathcal{P}_n$ 
		denotes the set of all homogeneous polynomials $p(z,w)$ of degree $n$.  
		Each subspace $\mathcal{P}_n$ reduces $T$ and 
		$T|_{\mathcal{P}_n} \cong J_n(0)$.
	\end{Example}
	
	Having briefly explored the interplay between tensor products and complex
	symmetric operators, we come to the following definition.

	\begin{Definition}
		An operator $A$ in $B(\h)$ is called an \emph{indestructible} complex
		symmetric operator if $A \otimes B$ is a complex symmetric operator on $\h\otimes\K$
		for all $B$ in $B(\K)$.
	\end{Definition}

	Let us note that an indestructible complex symmetric operator must
	indeed be complex symmetric since $A \otimes 1 \cong A$. 
	Clearly, indestructibility is a rather strong property.  In fact, from the definition alone, 
	it is not immediately clear whether any nonzero examples exist.  
	As we will see, the nonzero indestructible complex symmetric
	operators are precisely those operators which are nilpotent of order two.
	
	\begin{Theorem}
		$T$ is an indestructible complex symmetric operator 
		if and only if $T$ is nilpotent of order $\leq 2$.
	\end{Theorem}

	\begin{proof}
		If $A$ is nilpotent of order $\leq 2$, then $A \otimes B$ is 
		also nilpotent of order $\leq 2$.
		By Lemma~\ref{TheoremCSO}, $A \otimes B$ is 
		complex symmetric whence $A$ is indestructible.

		Before embarking on the remaining implication, let us first 
		remark that if $T$ is $C$-symmetric, then
		\begin{equation}\label{eq:word}
			w(T,T^*)=Cw(T^*, T)C
		\end{equation}
		holds for each word $w(x,y)$ in the noncommuting variables $x,y$.
		This fact will be useful in what follows.	
		
		Now suppose that $A$ is an indestructible complex
		symmetric operator. 
		For any other operator $B$ and any word 
		$w(x,y)$, we obtain
		\begin{align*}
			\norm{w(A,A^*)}\norm{w(B,B^*)}
			&= \norm{ w(A,A^*)\otimes w(B,B^*)} \\
			&= \norm{ w(A \otimes B, A^*\otimes B^*) } \\
			&= \norm{ w(A^*\otimes B^*, A\otimes B) } && \text{by \eqref{eq:word}}\\
			&=\norm{ w(A^*,A) } \norm{w(B^*,B)}.
		\end{align*}
		Since $A$ is complex symmetric we apply \eqref{eq:word} again to obtain
		\begin{equation}\label{eq:eq of norms}
			\|w(A, A^*)\| \|w(B, B^*)\|= \|w(A, A^*)\| \|w(B^*, B)\|.
		\end{equation}
		Letting $w(x,y)=yx^2$ and 
		\begin{equation*}
			B = \megamatrix{0}{\alpha}{0}{0}{0}{\beta}{0}{0}{0},
		\end{equation*}
		where $\alpha,\beta$ are non-negative real numbers, a simple
		computation reveals that
		\begin{equation*}
			\norm{w(B, B^*)} = \alpha^2 \beta,\qquad
			\norm{w(B^*, B)} = \alpha \beta^2.
		\end{equation*}
		If $\alpha\not=\beta$, then \eqref{eq:eq of norms}  implies that
		$\norm{w(A, A^*)}= 0$ so that $(A^*A)A = 0$.  Therefore $\ran A \subseteq
		\ker A^*A = \ker A$ whence $A^2 = 0$, as desired.
	\end{proof}

\section{Unitary equivalence to a truncated Toeplitz operator}\label{se:tto}

	The study of truncated Toeplitz operators 
	has been largely motivated by a seminal paper of Sarason
	\cite{Sarason}.  We briefly recall the basic definitions, referring the reader
	to the recent survey article \cite{RPTTO} for a more thorough introduction.

	In the following, $H^2$ denotes the classical Hardy space on the
	open unit disk.  For each nonconstant inner function $u$, we consider the 
	corresponding \emph{model space} $\K_u := H^2 \ominus u H^2$.
	Letting $P_u$ denote
	the orthogonal projection from $L^2$ onto $\K_u$, for each
	$\phi$ in $L^{\infty}$ we define the 
	\emph{truncated Toeplitz operator} $A_{\phi}^u : \K_u\to\K_u$ by
	setting 
	\begin{equation*}
		A_{\phi}^uf = P_u(\phi f).
	\end{equation*}
	Each such operator is $C$-symmetric with respect to the conjugation
	$Cf = \overline{fz}u$ on $\K_u$.
	We say that $A_{\phi}^u$ is an \emph{analytic} truncated Toeplitz operator
	if the \emph{symbol} $\phi$ belongs to $H^{\infty}$, in which case $A_{\phi}^u = \phi(A_z^u)$
	by the $H^{\infty}$-functional calculus for $A_z^u$.
	
	A significant amount of evidence has been accumulated which indicates
	that truncated Toeplitz operators provide concrete models for general complex symmetric operators
	\cite{TTOSIUES, ATTO, STZ}.  In fact, a surprising array of complex symmetric operators
	can be shown to be unitarily equivalent to truncated Toeplitz operators.  These results have led to
	several open problems and conjectures \cite[Question 5.10]{TTOSIUES}, \cite[Sect.~7]{ATTO}. We refer
	the reader to \cite[Sect.~9]{RPTTO} for a thorough discussion of the topic.

	By Lemma~\ref{TheoremCSO}, we know that operators which are nilpotent of
	order two are complex symmetric.   We now go a step further
	and prove that every such operator is
	unitary equivalent to an analytic truncated Toeplitz operator.
		
	Before proceeding, we require a few words about Hankel operators.
	First let us recall that the \emph{Hankel operator} $H_{\phi}:H^2 \to H^2_-$ with \emph{symbol} $\phi$ in $L^{\infty}$
	is the linear operator defined by $H_{\phi}f = P_-(\phi f)$, where $P_-$ denotes
	the orthogonal projection from $H^2$ onto $H^2_-:=L^2 \ominus H^2$.
	A detailed treatise on the subject of Hankel operators is \cite{Peller}.  We refer
	the reader there for a complete treatment of the subject.

	The first result required is the well-known relationship \cite[Ch.~1, eq.~(2.9)]{Peller}
	\begin{equation}\label{eq:Hankel}
		A_{\phi}^u = M_u H_{\overline{u}\phi}|_{\K_u},
	\end{equation}
	where $u$ is an inner function and $\phi$ belongs to $H^{\infty}$.
	The next ingredient is \cite[Ch.~1, Thm.~2.3]{Peller}.
	
	\begin{Lemma}\label{LemmaSymbol}
		For $\psi$ in $L^{\infty}$, the following are equivalent:
		\begin{enumerate}\addtolength{\itemsep}{0.35\baselineskip}
			\item $\ker H_{\psi}$ is nontrivial,
			\item $\ran H_{\psi}$ is not dense in $H^2_{-}$,
			\item $\psi = \overline{u}\phi$ for some inner function $u$ and
				some $\phi$ in $H^{\infty}$.
		\end{enumerate}
	\end{Lemma}

	Finally, we need the following deep result from \cite{Treil}
	(see \cite[Ch.~12, Thm.~8.1]{Peller}):
	
	\begin{Lemma}\label{LemmaModulus}
		If $A \geq 0$ is an operator on a separable, infinite-dimensional
		Hilbert space, then the following are equivalent:
		\begin{enumerate}\addtolength{\itemsep}{0.35\baselineskip}
			\item $A$ is unitarily equivalent to the modulus of a Hankel
				operator,
			\item $A$ is unitarily equivalent to the modulus of a
				self-adjoint Hankel operator,
			\item $A$ is not invertible, and $\ker A$ is either trivial or
				infinite-dimensional.
		\end{enumerate}
	\end{Lemma}

	The following general result 
	may be of independent interest.
	
	\begin{Lemma}\label{le:modulus}
		Any positive operator is unitarily equivalent to the modulus of an analytic
		truncated Toeplitz operator.
	\end{Lemma}

	\begin{proof}
		Suppose $B \geq 0$ and consider the operator $B'=B\oplus 0$, where 0 acts on
		an infinite dimensional Hilbert space. By Lemma~\ref{LemmaModulus}, $B'$ is
		unitarily equivalent to the modulus $|H_{\psi}|$ of some Hankel operator $H_\psi:H^2\to H^2_-$.  In light
		of Lemma~\ref{LemmaSymbol}, it follows that $\psi = \overline{u}\phi$ for 
		some inner function $u$ and some $\phi$ in $H^{\infty}$. 
		We may assume that $u$ and $\phi$ are coprime, since a common inner factor
		of both would cancel in the evaluation of $\psi = \overline{u}\phi$.  
		
		By \cite[Ch.1, Thm~2.4]{Peller},
		the restriction $$\hat H:\K_u\to H^2_-\ominus \overline{u}H^2_-$$ of 
		$H_{\overline{u}\phi}$ to $\K_u$ is injective and has dense range.
		In other words, $|\hat H|$ is unitarily equivalent to
		$B|_{(\ker B)^{\perp}}$.  The operator 
		\begin{equation*}
			W:H^2_-\ominus \overline{u}H^2_-\to \K_u
		\end{equation*}
		defined by $W = M_u|_{H^2_-\ominus \overline{u}H^2_-}$ is unitary and
		satisfies $A^u_\phi=W\hat H$ by \eqref{eq:Hankel}.  Therefore
		\begin{equation*}
			|A^u_\phi| \cong | \hat H|  \cong B|_{(\ker B)^{\perp}}.
		\end{equation*}

		Now let $v$ be an inner function such that $\dim \K_v = \dim \ker B$
		(e.g., a Blaschke product with $\dim \ker B$ zeros).  Noting that 
		\begin{equation*}
			\K_{uv} = K_u\oplus u K_v = vK_u\oplus K_v,
		\end{equation*}
		we see that the matrix of
		\begin{equation*}
			A^{uv}_{v\phi}:K_u\oplus u K_v\to  vK_u\oplus K_v
		\end{equation*}
		with respect to the decompositions above is
		\begin{equation*}
			\minimatrix{vA^u_{\phi}}{0}{0}{0}.
		\end{equation*}
		Since $\dim \K_v = \dim \ker B$, we conclude that
		\begin{equation*}
			|A^{uv}_{v\phi}|
			\cong |A_{\phi}^u| \oplus 0
			\cong B|_{ (\ker B)^{\perp}} \oplus B|_{\ker B} = B.\qedhere
		\end{equation*}
	\end{proof}
	
	Armed with the preceding lemma, we are ready to prove the following theorem.

	\begin{Theorem}\label{th:main}
		If $T$ is nilpotent of order two, then $T$ is unitarily
		equivalent to an analytic truncated Toeplitz operator.
	\end{Theorem}
	
	\begin{proof}
		It follows from \eqref{eq:form nilp 2} that any nilpotent 
		of order two is unitarily equivalent to an operator of the form
		\begin{equation*}
			N = \megamatrix{0}{0}{0}{0}{0}{0}{B}{0}{0}
		\end{equation*}
		on $\h\oplus \h'\oplus \h$, where $\h, \h'$ are Hilbert spaces and $B\geq 0$ acts
		on $\h$. By Lemma~\ref{le:modulus}, we may assume $\h=\K_u$ and that
		$B=|A^{u}_{\phi}|$ for some inner function $u$ and $\phi$ in $H^{\infty}$.
		Let ${v}$ be an inner function such that $\dim \K_{v}  =\dim \h'$
		and let $\omega:\h'\to \K_{v}$ be unitary.  With respect to the
		decomposition
		\begin{equation*}
			\K_{u ^2{v}  }=\K_u \oplus u \K_{v}  \oplus u
			{v}   \K_u 
		\end{equation*}
		we have
		\begin{equation*}
			A^{u ^2{v}  }_{u {v}  \phi}=
			\begin{bmatrix}
				0&0&0\\0&0&0\\u {v}   A^u _\phi &0 &0
			\end{bmatrix}.
		\end{equation*}
		Since $A^u _\phi$ is complex symmetric, we can write
		$A^u _\phi=V|A^u _\phi|$ with $V$ unitary \cite[Cor.~1]{G-P-II}. If 
		$W$ denotes the unitary operator
		\begin{equation*} 
			(I_{\K_u
			}\oplus \omega^*\bar u  \oplus V^*\bar u  \bar {v}  ): \K_u
			\oplus u\K_{v}  \oplus u {v}   \K_u \to \K_u
			\oplus \h'\oplus \K_u ,
		\end{equation*}
		then
		\begin{equation*}
			W (A^{u ^2{v}  }_{u {v}  \phi}) W^*=N,
		\end{equation*}
		which proves the theorem.
	\end{proof}

	\begin{Example}
		If $A$ is a noncompact operator on $\h$, then the operator
		\begin{equation*}
			T = \minimatrix{0}{0}{A}{0}
		\end{equation*}
		on $\h\oplus \h$ is unitarily equivalent to an analytic truncated Toeplitz operator $A_{\phi}^u$
		However, since any truncated Toeplitz operator whose symbol is continuous on the unit circle must be of the form
		normal plus compact  \cite[Thm.~1, Cor.~2]{CAGTTO}, $\phi$ cannot be continuous.
	\end{Example}

\section{Open questions}

	We conclude this note with some questions suggested by the preceding work.
	
	\begin{Question}
		Formula~\eqref{eq:word} may be generalized by considering polynomials
		$p(x,y)$ in two noncommuting
		variables $x,y$. Then
		\begin{equation*}
			p(T,T^*) = C\widetilde{p}(T^*,T)C
		\end{equation*}
		where $\widetilde{p}(x,y)$ is obtained from $p(x,y)$ by conjugating each
		coefficient. If $T$ is complex symmetric, it follows then that 
		\begin{equation}\label{eq-Polynomial}
			\norm{p(T,T^*)} = \norm{ \widetilde{p}(T^*,T)}
		\end{equation}
		holds for every $p(x,y)$.
		 Does the converse hold? That is, if $T$ in
		$B(\h)$ satisfies \eqref{eq-Polynomial} for every
		polynomial $p(x,y)$ in two noncommuting variables $x,y$, does it
		follow that 
		$T$ is a complex symmetric operator?
	\end{Question}

	Note that considering only \emph{words} in $T$ and $T^*$ is not sufficient to 
	characterize complex symmetric operators.  Indeed, if $S$ denotes the
	unilateral shift, then it is easy to see that
	$\norm{w(S,S^*)} = \norm{\widetilde{w}(S^*,S)} = 1$ for any word $w(x,y)$.

	The following question stems from the proof of Theorem~\ref{th:main}.
	
	\begin{Question}
		If $T$ is unitarily equivalent to a truncated Toeplitz operator, then does the operator 
		$T \oplus 0$ have the same property?
	\end{Question}

	Although partial results in this direction appear in \cite[Section 6]{STZ},
	the preceding question appears troublesome even in low dimensions.  

\bibliography{TRNOOT}

\end{document}